\documentclass{article}
\usepackage{mathrsfs, amsmath, amsthm}
\usepackage{subfig, graphicx, color}
\usepackage{array}
\usepackage{cases}
\makeatletter
\def\th@plain{%
  \upshape 
}
\makeatother

\makeatletter
\renewenvironment{proof}[1][\proofname]{\par
  \pushQED{\qed}%
  \normalfont \topsep6\p@\@plus6\p@\relax
  \trivlist
  \item[\hskip\labelsep
        \bfseries
    #1\@addpunct{.}]\ignorespaces
}{%
  \popQED\endtrivlist\@endpefalse
}
\makeatother

\newtheorem{theorem}{Theorem}
\numberwithin{theorem}{section}
\newtheorem{lemma}{Lemma}
\newtheorem{corollary}{Corollary}

\newtheorem*{conjecture*}{Conjecture}

\newtheorem{claim}{Claim}

\theoremstyle{definition}

\usepackage[left=25mm,top=25mm,bottom=25mm,right=25mm]{geometry}
\usepackage[T1]{fontenc}
\usepackage[varg]{txfonts}

\usepackage{enumitem}
\usepackage[square, numbers, sort&compress]{natbib}

\ifx\pdfoutput\undefined
 \usepackage[dvipdfm,%
  bookmarks=true,%
  bookmarksnumbered=true, 
  bookmarksopen=true, 
  plainpages=false,%
  pdfpagelabels,%
  colorlinks=true, 
  linkcolor=blue, 
  citecolor=blue,%
  hyperindex=true,
  urlcolor=black,
  pdfborder=001]{hyperref}
\else
 \usepackage[pdftex,%
  bookmarks=true,%
  bookmarksnumbered=true, 
  bookmarksopen=true, 
  plainpages=false,%
  pdfpagelabels,%
  colorlinks=true, 
  linkcolor=black, 
  citecolor=black,%
  anchorcolor=green,
  urlcolor= blue,
  breaklinks=true，
  hyperindex=true,
  pdfborder=001]{hyperref}
\fi

\newcounter{Hcase}
\newcounter{Hclaim}

\newcommand{\etal}{et~al.\ }

\def\int(#1){\mathrm{int}(#1)}
\def\ext(#1){\mathrm{ext}(#1)}
\def\Int(#1){\mathrm{Int}(#1)}
\def\Ext(#1){\mathrm{Ext}(#1)}
\def\mad(#1){\mathrm{mad}(#1)}
\def\la(#1){\mathrm{la}(#1)}
\newcommand{\Lfloor}{\left\lfloor}
\newcommand{\Rfloor}{\right\rfloor}

\newtheorem*{TCC}{Total Coloring Conjecture}
\begin{document}%
\title{Minimal counterexamples and discharging method}
\author{Tao Wang\footnote{{\tt Corresponding
author: wangtao@henu.edu.cn}}\\
{\small Institute of Applied Mathematics}\\
{\small College of Mathematics and Information Science}\\
{\small Henan University, Kaifeng, 475004, P. R. China}}
\date{March 2, 2014}
\maketitle

\begin{abstract}%
Recently, the author found that there is a common mistake in some papers by using minimal counterexample and discharging method. We first discuss how the mistake is generated, and give a method to fix the mistake. As an illustration, we consider total coloring of planar or toroidal graphs, and show that: if $G$ is a planar or toroidal graph with maximum degree at most $\kappa - 1$, where $\kappa \geq 11$, then the total chromatic number is at most $\kappa$.   
\end{abstract}
\section{Introduction}

A graph property $\mathcal{P}$ is {\em deletion-closed} if $\mathcal{P}$ is closed under taking subgraphs. We denote the minimum degree and maximum degree of a graph $G$ by $\delta(G)$ and $\Delta(G)$, respectively. We denote $\xi(G)$ a parameter of $G$, such as total chromatic number, list chromatic index, list total chromatic number, and so on. We denote $\zeta(G)$ a function of $\Delta(G)$, and denote $\lambda_{1}, \lambda_{2}, \kappa$ positive integers. Most of the results regarding planar graphs or toroidal graphs were proved by taking a minimal counterexample and using discharging method. Recently, the author found that there are many papers investigated results in the following form.

\begin{enumerate}[label = ($\ast$)]%
\item Let $G$ be a planar or toroidal graph with deletion-closed property $\mathcal{P}$. If $\Delta(G) \geq \lambda_{1}$, then $\xi(G) \leq \zeta(G)$. \qed
\end{enumerate}
In the proof, they wrote ``Let $G$ be a minimal counterexample. By the minimality of $G$, we have that $\xi(G - e) \leq \zeta(G)$.'' But something has been ignored, thus the argument is wrong because we cannot guarantee $\Delta(G - e) \geq \lambda_{1}$, that is, the condition $\Delta(G) \geq \lambda_{1}$ is not deletion-closed, so we cannot use the minimality of $G$. Therefore, some researchers changed to prove the corresponding results in the following form. 

\begin{enumerate}[label = ($\diamondsuit$)]%
\item Let $G$ be a planar or toroidal graph with deletion-closed property $\mathcal{P}$. If the maximum degree is at most $\lambda_{2}$, where $\lambda_{2} \geq \lambda_{1}$, then $\xi(G) \leq \zeta(G)$. \qed
\end{enumerate}

Hence, most of proofs about planar graphs can be fixed by changing the statement to the above form ($\diamondsuit$). But for the toroidal graphs, most of the proofs cannot be fixed even you adopt the above form ($\diamondsuit$). In the proof, to derive a contradiction, after the discharging process, we need to show that at least one element (vertex/face) has positive final charge. The common doing is to show the final charge of $\lambda_{2}$-vertex is positive, but maybe $\Delta(G) < \lambda_{2}$ and there is no $\lambda_{2}$-vertex.

Until now, the author found the results in \cite{MR1724810, MR3146839, MR2679093, MR3146527, MR2464909, MR1781294, MR2155783, MR2907325} and the corollaries in \cite{MR2458420, MR2552636, MR2277573, MR2475012, MR2118301} are wrong. To the author's knowledge, the earliest paper having this problem is Zhao's paper \cite{MR1724810} on total coloring, thus we only consider the total coloring problem. 

A {\em total coloring} of a graph $G$ is an assignment of colors to the vertices and edges of $G$ such that every pair of adjacent/incident elements receive distinct colors. The {\em total chromatic number} of a graph $G$, denoted by $\chiup''(G)$, is the minimum number of colors in a total coloring of $G$. It is obvious that the total chromatic number of a graph $G$ has a trivial lower bound $\Delta(G) + 1$. For the upper bound, Behzad \cite{Behzad1965} raised the following well-known Total Coloring Conjecture (TCC):

\begin{TCC}
Every graph with maximum degree $\Delta$ admits a total coloring with at most $\Delta + 2$ colors.
\end{TCC}

The conjecture was verified in the case $\Delta =3$ by Rosenfeld \cite{MR0278995} and Vijayaditya \cite{MR0285447} independently and also by Yap \cite{MR0976059}. It was confirmed in the case $\Delta \in \{4, 5\}$ by Kostochka \cite{MR0453576, MR1425788}, in fact the proof holds for multigraphs. Regarding planar graphs, the conjecture was verified in the case $\Delta \geq 9$ by Borodin \cite{MR977440} and in the case $\Delta = 7$ by Sanders and Zhao \cite{MR1684286}; the case $\Delta = 8$ was a consequence of Vizing's theorem about planar graphs \cite{Vizing1965} and four coloring theorem (for more details, see Jensen and Toft \cite{MR1304254}). Thus, the only remaining case for planar graphs is that of maximum degree six. Note that best known upper bound on the total chromatic number of planar graph with maximum degree $6$ is $9$ \cite{MR977440}. 

For planar graphs with large maximum degree, the total chromatic number can be obtained. Precisely, Borodin \cite{MR977440} showed that if $\Delta \geq 14$ then $\chiup''(G) = \Delta(G) + 1$. Borodin, Kostochka and Woodall improved the result to the case $\Delta \geq 12$ \cite{MR1483474} and $\Delta = 11$ \cite{MR1464341}. Recently, Wang \cite{MR2285452} further improved the result for $\Delta = 10$, and Kowalik \etal \cite{MR2448906} improved the result for $\Delta = 9$.

In section~\ref{Sec:structure}, we give some structural results which are very helpful in the proof of total coloring problem. In section~\ref{Sec:torus}, we give an illustration how to prove the statement in the revised form ($\diamondsuit$).

\section{Total $\kappa$-coloring}\label{Sec:structure}
A {\em $\kappa$-deletion-minimal} graph with respect to total coloring, is a graph with maximum degree at most $\kappa - 1$ such that its total chromatic number is greater than $\kappa$, but the total chromatic number of every proper subgraph is at most $\kappa$. In this section, we give many structural results on $\kappa$-deletion-minimal graph $G$, most of which can be obtained by trivially extending the corresponding proofs in other papers. Note that some of the results in this section may be not used in section~3, and we just collect as many results as possible. All the solid black dots are only incident with the edges depicted in the figures. 

Usually, we first give a partial total coloring of $G$, and then we extend the coloring to $G$ in the proof. Since an uncolored vertex with degree at most $\Lfloor \frac{\kappa - 1}{2} \Rfloor$ forbids at most $2\deg \leq 2\Lfloor \frac{\kappa - 1}{2} \Rfloor \leq \kappa - 1$ colors, so we always have at least one available color for the vertex, thus we will not care about the coloring of the vertices with degree at most $\Lfloor \frac{\kappa - 1}{2} \Rfloor$. 

A vertex of degree $\tau$, at most $\tau$ and at least $\tau$ are called a $\tau$-vertex, $\tau^{-}$-vertex and $\tau^{+}$-vertex, respectively. Let $[\kappa]$ denote the set $\{1, 2, \dots, \kappa\}$. We denote $\mathcal{U}(w)$ the set of colors which are assigned to the vertex $w$ and edges incident with $w$. 

\begin{lemma}\label{2-connected}%
The graph $G$ is $2$-connected. 
\end{lemma}

\begin{lemma}[Wang \cite{2012arXiv1206.3862W}]\label{DEDGE}%
If $u$ and $v$ are two adjacent vertices with $\deg_{G}(v) \leq \Lfloor \frac{\kappa - 1}{2} \Rfloor$, then $\deg_{G}(u) + \deg_{G}(v) \geq \kappa + 1$.
\end{lemma}

\begin{lemma}
The minimum degree is at least $\kappa - \Delta(G) + 1$. 
\end{lemma}

\begin{lemma}%
If $\kappa \geq 5$, then the subgraph induced by the edges incident with $2$-vertices is a forest.
\end{lemma}
\begin{proof}
Firstly, by \autoref{DEDGE}, the set of $2$-vertices is independent and the edge induced subgraph is bipartite. Suppose that it contains a cycle $C$. By the minimality of $G$, the graph $G - E(C)$ admits a total coloring $\varphi$ with at most $\kappa$ colors. We can extend $\varphi$ to $G$ by using the known result that every even cycle is $2$-edge-choosable, which leads to a contradiction.  
\end{proof}

\begin{lemma}\label{SumDelta3}%
Let $u$ and $v$ be two adjacent vertices with $\deg_{G}(v) \leq \Lfloor \frac{\kappa - 1}{2} \Rfloor$ and $\deg_{G}(u) + \deg_{G}(v) \leq \kappa + 1$. If $uv$ is contained in a triangle $uvw$, then $\deg_{G}(w) = \kappa - 1$.
\end{lemma}
\begin{proof}%
  By contradiction, suppose that $\deg_{G}(w) \leq \kappa -2$. By the minimality of $G$, the graph $G - uv$ admits a total coloring with at most $\kappa$ colors. Now, we erase the color on the vertex $v$, and denote the resulting coloring by $\varphi$. If $\{1, 2, \dots, \kappa\}$ is not the union of $\mathcal{U}_{\varphi}(u)$ and $\mathcal{U}_{\varphi}(v)$, then we can extend the coloring $\varphi$ to $uv$. Hence, the set $\{1, 2, \dots, \kappa\}$ is the union of $\mathcal{U}_{\varphi}(u)$ and $\mathcal{U}_{\varphi}(v)$; in fact, it is the disjoint union of $\mathcal{U}_{\varphi}(u)$ and $\mathcal{U}_{\varphi}(v)$ since $|\mathcal{U}_{\varphi}(u)| + |\mathcal{U}_{\varphi}(v)| = \kappa$. Note that $\varphi(wv) \notin \mathcal{U}_{\varphi}(u)$. Let $\phi$ be the coloring from $\varphi$ by assigning the color $\varphi(wv)$ to $uv$ and erasing the color on $wv$. Similarly, we can prove that $\{1, 2, \dots, \kappa\}$ is the union (not necessarily disjoint union) of $\mathcal{U}_{\phi}(w)$ and $\mathcal{U}_{\varphi}(v)$. Therefore, we have $\mathcal{U}_{\varphi}(u) \subseteq \mathcal{U}_{\phi}(w) \subset \mathcal{U}_{\varphi}(w)$. Since $\deg_{G}(w) \leq \kappa - 2$, it follows that there exists a color $\alpha \notin \mathcal{U}_{\varphi}(u) \cup \mathcal{U}_{\varphi}(w)$. Note that $\varphi(uw) \notin \mathcal{U}_{\varphi}(v)$. We extend $\varphi$ by assigning $\alpha$ to $uw$ and assigning $\varphi(uw)$ to $uv$. 
\end{proof}

\begin{lemma}\label{2Tau}%
If $\kappa = 2\tau$, then $(\tau, \tau)$-edge is not contained in a triangle.
\end{lemma}
\begin{proof}%
(See also Lemma 2~(vi) in \cite{MR2448906}) By contradiction, suppose that a $(\tau, \tau)$-edge $uv$ is contained in a triangle $uvw$. By the minimality of $G$, the graph $G - uv$ admits a total coloring with at most $\kappa$ colors. Now, we erase the color on the vertex $v$, and denote the resulting coloring by $\varphi$. Let $A_{\varphi}(uv)$ be the set of available colors for the edge $uv$, and $A_{\varphi}(v)$ the set of available colors for the vertex $v$. If there exist $\alpha_{1} \in A_{\varphi}(uv)$ and $\alpha_{2} \in A_{\varphi}(v)$ such that $\alpha_{1} \neq \alpha_{2}$, then we can extend $\varphi$ by assigning $\alpha_{1}$ to $uv$ and $\alpha_{2}$ to $v$. So we may assume that $A_{\varphi}(uv) = A_{\varphi}(v) = \{\alpha\}$. Hence, we have that $|\mathcal{U}_{\varphi}(u) \cup \mathcal{U}_{\varphi}(v)| = \kappa - 1$ and $\mathcal{U}_{\varphi}(u) \cap \mathcal{U}_{\varphi}(v) = \emptyset$. Exchanging the colors on $wu$ and $wv$, and assigning $\alpha$ to $uv$ and $\varphi(wv)$ to $v$.
\end{proof}

\begin{lemma}\label{OneSmall}%
Let $w$ be a vertex with $\{w_{1}, w_{2}, w_{3}\} \subseteq N_{G}(w)$ and $w_{1}w_{2} \in E(G)$. If $\deg(w) + \deg(w_{1}) = \kappa + 1$ and $\deg(w_{1}) \leq \left \lfloor\frac{k-1}{2} \right \rfloor$, then $\deg(w) + \deg(w_{3}) \geq \kappa + 2$. 
\end{lemma}
\begin{proof}%
Suppose that $\deg(w) + \deg(w_{3}) \leq \kappa + 1$, which implies that $\deg(w_{3}) \leq \left \lfloor\frac{k-1}{2} \right \rfloor$. The graph $G - ww_{1}$ admits a total coloring with at most $\kappa$ colors. Now, we erase the colors on the vertices $w_{1}$ and $w_{3}$, and denote the resulting coloring by $\varphi$. Notice that $\{1, 2, \dots, \kappa\}$ is the disjoint union of $\mathcal{U}_{\varphi}(w)$ and $\mathcal{U}_{\varphi}(w_{1})$. Notice also that $\mathcal{U}_{\varphi}(w) \cup \mathcal{U}_{\varphi}(w_{3}) = \{1, 2, \dots, \kappa\}$; otherwise, reassigning $\varphi(ww_{3})$ to $ww_{1}$ and assigning a color in $[\kappa] \setminus (\mathcal{U}_{\varphi}(w) \cup \mathcal{U}_{\varphi}(w_{3}))$ to $ww_{3}$. Note that $\varphi(ww_{2}) \notin \mathcal{U}_{\varphi}(w_{3})$. Now, exchanging the colors on $ww_{2}$ and $w_{1}w_{2}$, and reassigning $\varphi(ww_{2})$ to $ww_{3}$ and $\varphi(ww_{3})$ to $ww_{1}$.   
\end{proof}

\begin{lemma}\label{TwoTriangles}%
Let $ww_{2}$ be contained in two triangles $ww_{1}w_{2}$ and $ww_{2}w_{3}$. If $\deg(w) + \deg(w_{2}) = \kappa + 1$ and $\deg(w_{2}) \leq \left \lfloor\frac{k-1}{2} \right \rfloor$, then all the other neighbors of $w$ have degree at least $\deg(w_{2}) + 2$ or $\left \lfloor\frac{k+1}{2} \right \rfloor$. 
\end{lemma}
\begin{proof}%
Suppose that $w_{4}$ is a vertex with degree at most $\left \lfloor\frac{k-1}{2} \right \rfloor$. The graph $G - ww_{2}$ admits a total coloring with at most $\kappa$ colors. Now, we erase the colors on the vertices $w_{2}$ and $w_{4}$, and denote the resulting coloring by $\varphi$. Notice that $\{1, 2, \dots, \kappa\}$ is the disjoint union of $\mathcal{U}_{\varphi}(w)$ and $\mathcal{U}_{\varphi}(w_{2})$. Notice also that $\mathcal{U}_{\varphi}(w) \cup \mathcal{U}_{\varphi}(w_{4}) = \{1, 2, \dots, \kappa\}$; otherwise,  reassigning $\varphi(ww_{4})$ to $ww_{2}$ and assigning a color in $[\kappa] \setminus (\mathcal{U}_{\varphi}(w) \cup \mathcal{U}_{\varphi}(w_{4}))$ to $ww_{4}$. This implies that $\mathcal{U}_{\varphi}(w_{2}) \subset \mathcal{U}_{\varphi}(w_{4})$. Now, exchanging the colors on $ww_{1}$ and $w_{1}w_{2}$, and additionally exchanging the colors on $ww_{3}$ and $w_{3}w_{2}$, we obtain another partial total coloring $\sigma$. Similarly, we have that $\mathcal{U}_{\sigma}(w_{2}) \subset \mathcal{U}_{\sigma}(w_{4}) = \mathcal{U}_{\varphi}(w_{4})$, which implies that $\varphi(ww_{1}), \varphi(ww_{3}) \subseteq \mathcal{U}_{\varphi}(w_{4})$. Hence, we have that $\deg(w_{4}) \geq |\mathcal{U}_{\varphi}(w_{2})| + |\{\varphi(ww_{1}), \varphi(ww_{3}), \varphi(ww_{4})\}| = \deg(w_{2}) + 2$. 
\end{proof}

\begin{figure}%
\centering
\subfloat[]{\label{RC1}\includegraphics{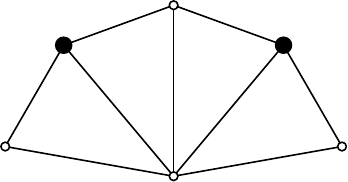}}\hfill~
\subfloat[]{\label{RC2}\includegraphics{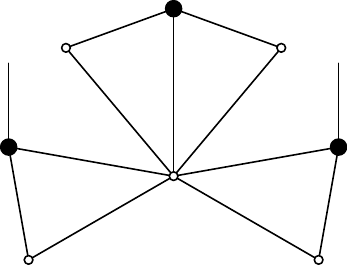}}\hfill~
\subfloat[]{\label{RC3}\includegraphics{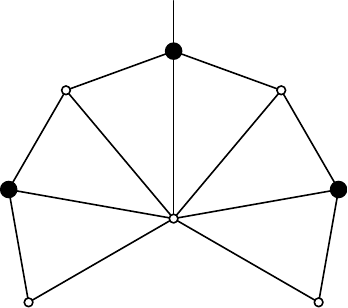}}\hfill~
\subfloat[]{\label{RC4}\includegraphics{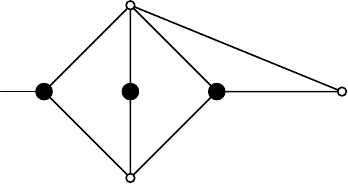}}\hfill~
\subfloat[]{\label{RC5}\includegraphics{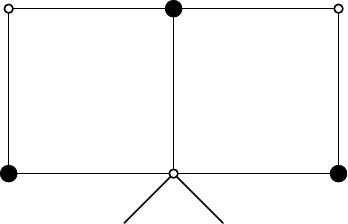}}\hfill~
\subfloat[]{\label{RC6}\includegraphics{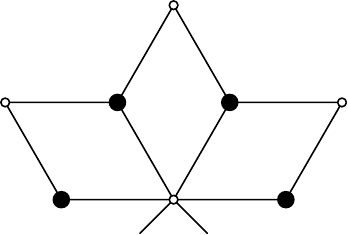}}\hfill~
\subfloat[]{\label{RC7}\includegraphics{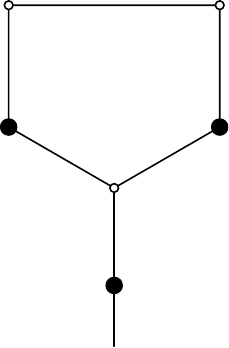}}\hfill~
\subfloat[]{\label{RC8}\includegraphics{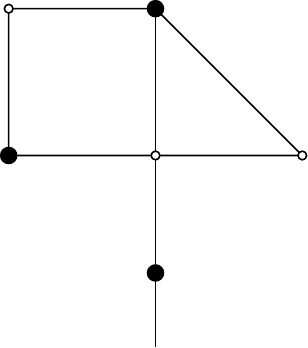}}
\caption{Reducible configurations}
\label{}
\end{figure}

\begin{figure}%
\centering
\subfloat[]{\label{RC9}\includegraphics{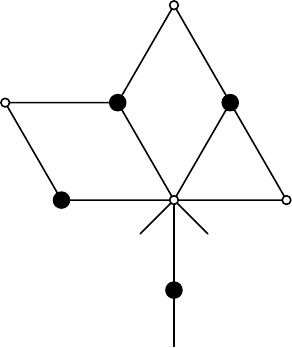}}\hfill~
\subfloat[]{\label{RC10}\includegraphics{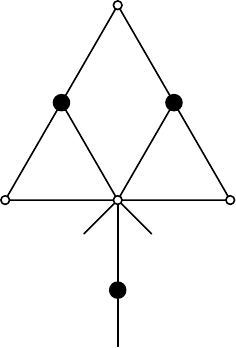}}\hfill~
\subfloat[]{\label{Net}\includegraphics{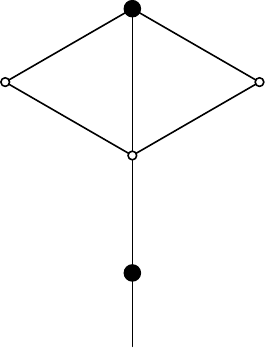}}
\caption{Reducible configurations}
\label{}
\end{figure}

\begin{lemma}%
Let $ww_{2}$ be contained in two triangles $ww_{2}w_{1}$ and $ww_{2}w_{3}$. If $\kappa \geq 7$ and $w_{1}$ is a $2$-vertex, then $\deg(w_{3}) \geq 4$.
\end{lemma}
\begin{proof}%
(See also Lemma 3~(iv) in \cite{MR2448906}). By contradiction, suppose that $w_{3}$ is a $3^{-}$-vertex. By the minimality of $G$, the graph $G - ww_{1}$ has a total coloring with at most $\kappa$ colors. Now, we erase the colors on the vertices $w_{1}$ and $w_{3}$, and denote the resulting coloring by $\varphi$. Thus, $[\kappa]$ is the disjoint union of $\varphi(w_{1}w_{2})$ and $\mathcal{U}_{\varphi}(w)$; otherwise, we can assign an available color to $ww_{1}$. If $\varphi(w_{1}w_{2}) \notin \mathcal{U}_{\varphi}(w_{3})$, then we recolor $ww_{1}$ and $ww_{3}$ with $\varphi(ww_{3})$ and $\varphi(w_{1}w_{2})$, respectively. Thus, we have $\varphi(w_{1}w_{2}) \in \mathcal{U}_{\varphi}(w_{3})$, but $\varphi(w_{1}w_{2}) \notin \{\varphi(w_{2}w_{3}), \varphi(ww_{3})\}$. Now, we recolor $ww_{1}, w_{1}w_{2}, ww_{2}$ and $ww_{3}$ with $\varphi(ww_{3}), \varphi(ww_{2}), \varphi(w_{1}w_{2})$ and $\varphi(ww_{2})$, respectively. 
\end{proof}
\begin{lemma}\label{No3*}%
If $\kappa \geq 7$ and $uvw$ is a $(2, \Delta, \Delta)$-triangle with $\deg(u) = 2$, then $v$ is not contained in another $(\Delta, 3, *)$-triangle. 
\end{lemma}
\begin{proof}%
See \cite[Lemma 3~(v)]{MR2448906}.
\end{proof}

\begin{lemma}\label{4Fan}%
If $\kappa \geq 7$, then the graph $G$ contains no configuration in \autoref{RC1}.
\end{lemma}
\begin{proof}%
See \cite[Lemma~6]{MR2448906}. 
\end{proof}

\begin{lemma}\label{FlightFan}%
If $\kappa \geq 7$, then the graph $G$ contains no configuration in \autoref{RC2}.
\end{lemma}
\begin{proof}
See \cite[Lemma~4]{MR2448906}. 
\end{proof}

\begin{lemma}%
If $\kappa \geq 9$, then the graph $G$ contains no configuration in \autoref{RC3}. 
\end{lemma}
\begin{proof}%
See \cite[Lemma~7]{MR2448906}.
\end{proof}

\begin{lemma}%
If $\kappa \geq 7$, then the graph $G$ contains no configuration in \autoref{RC4}. 
\end{lemma}
\begin{proof}%
See \cite{MR1464341} or \cite[Lemma 3~(vi)]{MR2448906}.
\end{proof}

\begin{lemma}[Shen and Yang \cite{MR2530185}]%
If $\kappa \geq 7$, then the graph $G$ contains no configurations in \autoref{RC5}, \ref{RC6}, \ref{RC7} and \ref{RC8}. 
\end{lemma}

\begin{lemma}[Du \etal \cite{MR2537481}]%
If $\kappa \geq 7$, then the graph $G$ contains no configurations in \autoref{RC9} and \autoref{RC10}.
\end{lemma}

\begin{lemma}\label{LNet}%
If $\kappa \geq 7$, then the graph $G$ contains no configuration in \autoref{Net}.
\end{lemma}
\begin{proof}%
Suppose that the edge $vv_{2}$ is contained in two triangles $vv_{1}v_{2}$ and $vv_{2}v_{3}$. We further assume that $v_{2}$ is a $3$-vertex and $v$ is adjacent to a $2$-vertex $u$. By the minimality of $G$, the graph $G - uv$ has a total coloring with at most $\kappa$ colors. We erase the colors on vertices $u$ and $v_{2}$, and denote the resulting coloring by $\varphi$. Without loss of generality, let $\varphi(vv_{1}) = 1$, $\varphi(vv_{2}) = 2$ and $\varphi(vv_{3}) = 3$. Note that $[\kappa]$ is the disjoint union of $\mathcal{U}_{\varphi}(v)$ and $\{\varphi(uw)\}$, where $w$ is the neighbor of $u$ other than $v$. Without loss of generality, we assume that $\varphi(uw) = \kappa$. If $\kappa \notin \{\varphi(v_{1}v_{2}), \varphi(v_{2}v_{3})\}$, then we recolor $vv_{2}$ with $\kappa$ and $uv$ with $2$. By symmetry, we assume that $\varphi(v_{1}v_{2}) = \kappa$. If $\varphi(v_{2}v_{3}) \neq 1$, then we recolor $vv_{1}, v_{1}v_{2}$ and $uv$ with $\kappa, 1$ and $1$, respectively. If $\varphi(v_{2}v_{3}) = 1$, then we recolor $vv_{1}, v_{1}v_{2}, v_{2}v_{3}, vv_{3}$ and $uv$ with $\kappa, 1, 3, 1$ and $3$, respectively. 
\end{proof}

\begin{lemma}%
If $\kappa \geq 7$, then $G$ contains no $(4, 4, 4)$-triangle. 
\end{lemma}
\begin{proof}%
Suppose that $uvw$ is a $(4, 4, 4)$-triangle. The graph $G - \{uv, vw, uw\}$ admits a total coloring with at most $\kappa$ colors. Now, we erase the colors on the vertices $u, v$ and $w$, and denote the resulting coloring by $\varphi$. Note that each element in $\{u, v, w, uv, vw, uw\}$ forbids at most four colors and each element has at least three available colors. Thus, we can extend $\varphi$ to $G$ by using the fact that every triangle is totally $3$-choosable \cite[Theorem 2.2]{MR1617950}. 
\end{proof}

\section{Total coloring of planar and toroidal graphs}\label{Sec:torus}
McDiarmid and S{\'a}nchez-Arroyo \cite{MR1210116} gave a general upper bound in terms of the maximum degree (the graph is not necessarily planar or toroidal). 
\begin{theorem}[\cite{MR1210116}]\label{UpperBound}%
If $G$ is a simple graph with maximum degree $\Delta$, then $\chiup''(G) \leq \frac{7}{5}\Delta + 3$. 
\end{theorem}
\begin{theorem}%
Let $G$ be a planar or toroidal graph with maximum degree at most $\kappa - 1$. If $\kappa \geq 11$, then $\chiup''(G) \leq \kappa$. 
\end{theorem}
\begin{proof}%
Let $G$ be a counterexample to the theorem with the minimum number of edges. Thus, it is a $\kappa$-deletion-minimal graph, and all the properties of $\kappa$-deletion-minimal graph hold for $G$. By \autoref{UpperBound}, we assume that $\Delta(G) \geq 7$. We also assume that $G$ has been embedded in the corresponding surface. Let $F(G)$ denote the face set of $G$. By \autoref{2-connected}, the graph $G$ is $2$-connected and $\delta(G) \geq 2$. The degree $\deg_{G}(f)$ of a face $f$ is the number of edges with which it is incident, and every cut edge being counted twice. 

\begin{claim}[Kowalik \etal \cite{MR2448906}]\label{One2-vertex}%
Every vertex is adjacent to at most one $2$-vertex.
\end{claim}
From Euler's formula, we have the following equality:
\begin{equation}%
\sum_{v \in V(G)}(\deg(v) - 4) + \sum_{f \in F(G)}(\deg(f) - 4) = -8 \mbox{ or } 0.
\end{equation}

Assign the initial charge of every vertex $v$ to be $\deg(v) - 4$ and the initial charge of every face $f$ to be $\deg(f) - 4$. We design appropriate discharging rules and redistribute charges among vertices and faces, such that the final charge of every vertex and every face is nonnegative; moreover, the final charge of every vertex with maximum degree is positive, which derives a contradiction.

A $2$-vertex is {\em good} if it is incident with a $5^{+}$-face, otherwise, it is {\em bad}.

{\bf The Discharging Rules:}
\begin{enumerate}[label= (R\arabic*)]%
\item\label{Triangle4-} If $f$ is a $(4^{-}, 8^{+}, 8^{+})$-face, then $f$ receives $\frac{1}{2}$ from each incident $8^{+}$-vertex. 
\item\label{Triangle5} If $f$ is a $(5, 7^{+}, 7^{+})$-face, then $f$ receives $\frac{1}{5}$ from the incident $5$-vertex and $\frac{2}{5}$ from each incident $7^{+}$-vertex.  
\item\label{Triangle6+} If $f$ is a $(6^{+}, 6^{+}, 6^{+})$-face, then $f$ receives $\frac{1}{3}$ from each incident vertex. 
\item\label{2-vertex} Every vertex sends $1$ to each adjacent bad $2$-vertex and $\frac{1}{2}$ to each adjacent good $2$-vertex.
\item\label{5+face2} Every $5^{+}$-face sends $1$ to each incident $2$-vertex.
\item\label{3-vertex} Every $3$-vertex receives $\frac{1}{3}$ from each adjacent vertex. 
\end{enumerate}

By \autoref{DEDGE} and the discharging rules, the final charge of every $3$-face is nonnegative. Clearly, the final charge of every $4$-face is zero. 

By \autoref{DEDGE} and \autoref{One2-vertex}, a $\tau$-face is incident with at most $\lfloor \frac{\tau}{3} \rfloor$ vertices of degree two. If $f$ is a $\tau$-face with $\tau \geq 5$, then the final charge is at least $\tau - 4 - \lfloor \frac{\tau}{3} \rfloor \times 1 \geq 0$ by \ref{5+face2}.

{\bf Let $v$ be a $2$-vertex.} If it is bad, then the final charge is $2 - 4 + 2 \times 1 = 0$ by \ref{2-vertex}. If it is good, then the final charge is at least $2 - 4  + 1 + 2 \times \frac{1}{2} = 0$ by \ref{2-vertex}. 

{\bf Let $v$ be a $3$-vertex.} The final charge is $3 - 4 + 3 \times \frac{1}{3} = 0$ by \ref{3-vertex}.

{\bf Let $v$ be a $4$-vertex.} Clearly, the final charge is equal to the initial charge zero.

{\bf Let $v$ be a $5$-vertex.} The final charge is at least $5 - 4 - 5 \times \frac{1}{5} = 0$ by \ref{Triangle5}.

{\bf Let $v$ be a $6$-vertex.} The final charge is at least $6 - 4 - 6 \times \frac{1}{3} = 0$ by \ref{Triangle6+}. 

{\bf Let $v$ be a $7$-vertex.} The final charge is at least $7 - 4 - 7 \times \frac{2}{5} > 0$ by \ref{Triangle5} and \ref{Triangle6+}. 

{\bf Let $v$ be an $8$-vertex.} The final charge is at least $8 - 4 - 8 \times \frac{1}{2} = 0$ by \ref{Triangle4-}, \ref{Triangle5} and \ref{Triangle6+}. Moreover, the final charge equals zero only if $v$ is incident with eight $3$-faces and every incident $3$-face contains a $4$-vertex, but this is impossible by \autoref{DEDGE} and \autoref{OneSmall}. Hence, the final charge of $v$ is positive.  

{\bf Let $v$ be a $t$-vertex with $9 \leq t \leq \kappa - 3$.} The final charge is at least $t - 4 - t \times \frac{1}{2} = \frac{t - 8}{2} > 0$ by \ref{Triangle4-}, \ref{Triangle5} and \ref{Triangle6+}. 

{\bf Let $v$ be a $(\kappa - 2)$-vertex.} Note that $v$ is not adjacent to $2$-vertices. If $v$ is not adjacent to $3$-vertices, then the final charge is at least $(\kappa-2) - 4 - (k-2) \times \frac{1}{2} = \frac{\kappa - 10}{2} > 0$. So we may assume that $v$ is adjacent to a $3$-vertex $u$. Suppose that $uv$ is contained in a triangle $uvw$. By \autoref{OneSmall}, the vertex $v$ is adjacent to exactly one $3$-vertex. Thus, the final charge is at least $(\kappa-2) - 4 - (k-2) \times \frac{1}{2} -\frac{1}{3} = \frac{\kappa - 10}{2} -\frac{1}{3} > 0$. So we may further assume that every edge $uv$ with $u$ is a $3$-vertex is not contained in a triangle. Thus, we have that the number of adjacent $3$-vertices and incident $3$-faces is at most $\kappa - 2$, and the final charge is at least $(\kappa-2) - 4 - (k-3) \times \frac{1}{2} -\frac{1}{3} = \frac{\kappa - 9}{2} -\frac{1}{3} > 0$. 

{\bf Let $v$ be a $(\kappa - 1)$-vertex.} 

A {\em fan} is a subgraph with some (at least one) consecutive $3$-faces such that $vv_{0}v_{1}, vv_{1}v_{2}, \dots, vv_{s-1}v_{s}$ are the boundaries and each of $vv_{0}$ and $vv_{s}$ is incident with a $4^{+}$-face. Let $\epsilon$ be the number of fans at $v$.
\begin{enumerate}[label = (\alph*)]%
\item The vertex $v$ is adjacent to a $2$-vertex. 

By \autoref{One2-vertex}, the vertex $v$ is adjacent to exactly one $2$-vertex $w$. First of all, suppose that $w$ is in a triangle. By \autoref{No3*}, no edge $uv$ with $u$ is a $3$-vertex is incident with $3$-faces. Thus, the number of adjacent $3$-vertices and incident $3$-faces is at most $\kappa - 2$, and thus the final charge is at least $(\kappa-1) - 4 - 1 - (k-2) \times \frac{1}{2} = \frac{\kappa - 10}{2} > 0$.

Next, we may assume that the $2$-vertex $w$ is not in a triangle. By \autoref{LNet}, each edge $uv$ with $u$ is a $3$-vertex is contained in at most one triangle, and thus each fan contains at most two $3$-vertices. Moreover, if a fan contains exactly two $3$-vertices, then the fan contains at least two $3$-faces by \autoref{DEDGE}. Note that $\epsilon$ is at most $\frac{\kappa - 2}{2}$ because the $2$-vertex $w$ is incident with two $4^{+}$-faces.  

Note that the number of adjacent $3$-vertices and incident $3$-faces is at most $(\kappa - 2) + \epsilon$; the number of $4^{+}$-face is at least $\epsilon + 1$. 

If $\kappa \geq 12$, then the final charge of $v$ is at least 
\[
(\kappa - 1) - 4 - 1 - (\kappa - 2 - \epsilon) \times \frac{1}{2} - 2\epsilon \times \frac{1}{3} = \frac{\kappa - 10}{2} - \frac{\epsilon}{6} \geq \frac{5\kappa - 58}{12} > 0.
\]

Now, we consider the case $\kappa = 11$. Note that $\epsilon \leq 4$. If $\epsilon = 4$, then at least three fans only contains one $3$-face and each such $3$-face contains at most one $3$-vertex, and then the final charge of $v$ is at least $10 - 4 - 1 - 5 \times \frac{1}{2} - 5 \times \frac{1}{3} > 0$. If $\epsilon = 3$, then the final charge of $v$ is at least $10 - 4 - 1 - 6 \times \frac{1}{2} - 6 \times \frac{1}{3} = 0$. Moreover, the final charge equals zero only if the local structure is as illustrated in \autoref{Ten} (note that the $2$-vertex $w$ is incident with two $4$-faces), but it is excluded by \autoref{RC4}. Hence, the final charge of $v$ is positive. 

If $\epsilon = 2$, then the final charge is at least $10 - 4 - 1 - 7 \times \frac{1}{2} - 4 \times \frac{1}{3} > 0$. 
If $\epsilon = 1$, then the final charge is at least $10 - 4 - 1 - 8 \times \frac{1}{2} - 2 \times \frac{1}{3} > 0$. If $\epsilon = 0$, then the final charge is at least $10 - 4 - 1 - 9 \times \frac{1}{3} > 0$.  

\begin{figure}%
\centering
\includegraphics{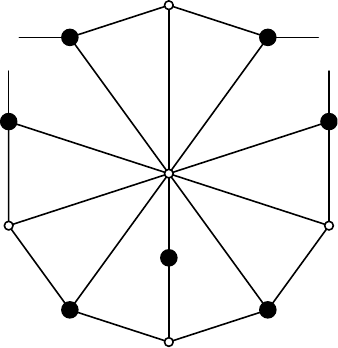}
\caption{}
\label{Ten}
\end{figure}

\item The vertex $v$ is not adjacent to $2$-vertices. 

If $v$ is not incident with $4^{+}$-faces, then according to \autoref{4Fan} and \autoref{FlightFan}, the vertex $v$ is adjacent to at most two $3$-vertices, and then the final charge is at least $(\kappa-1) - 4 - (\kappa-1) \times \frac{1}{2} - 2 \times \frac{1}{3} = \frac{\kappa - 9}{2} - \frac{2}{3} > 0$. So we may assume that $v$ is incident with at least one $4^{+}$-face.  
\begin{enumerate}[label = ($\bullet$)]%
\item Every fan contains at most three $3$-vertices by \autoref{4Fan} and \autoref{FlightFan}. 
\end{enumerate}
If $\epsilon = 0$, then the final charge is at least $(\kappa-1) - 4 - (\kappa-1) \times \frac{1}{3} = \frac{2\kappa - 14}{3} > 0$. If $\epsilon = 1$, then the final charge is at least $(\kappa-1) - 4 - (\kappa-2) \times \frac{1}{2} - 3 \times \frac{1}{3} = \frac{\kappa - 10}{2} > 0$. If $\epsilon = 2$, then according to ($\bullet$), \autoref{4Fan} and \autoref{FlightFan}, the vertex $v$ is adjacent to at most five $3$-vertices on the two fans, and then the final charge is at least $(\kappa-1) - 4 - (\kappa-3) \times \frac{1}{2} - 5 \times \frac{1}{3} = \frac{\kappa - 7}{2} - \frac{5}{3} > 0$. If $\epsilon \geq 3$ and every fan at $v$ contains at most two $3$-vertices, then the final charge of $v$ is at least $(\kappa-1) - 4 - (\kappa - 1 - \epsilon) \times \frac{1}{2} - 2\epsilon \times \frac{1}{3} = \frac{\kappa - 9}{2} - \frac{\epsilon}{6} \geq \frac{\kappa - 9}{2} - \frac{\kappa - 1}{12} > 0$. In the next, we assume that $\epsilon \geq 3$ and there exists a fan containing exactly three $3$-vertices. By \autoref{4Fan} and \autoref{FlightFan}, all the other fans contain at most two $3$-vertices, thus $v$ is adjacent to at most five $3$-vertices on fans. Hence, the final charge of $v$ is at least $(\kappa-1) - 4 - (\kappa - 1 - \epsilon) \times \frac{1}{2} - 5 \times \frac{1}{3} = \frac{\kappa - 9 + \epsilon}{2} - \frac{5}{3} > 0$. 
\end{enumerate}

Now, we have checked that the final charge of every vertex and every face is nonnegative. Let $w$ be a vertex with maximum degree. Clearly, the vertex $w$ is a $7^{+}$-vertex. From the above arguments, we have that $w$ has positive final charge, thus the sum of the final charge of every element is positive, which leads to a contradiction.
\end{proof}

\begin{corollary}%
If $G$ is a planar or toroidal graph with maximum degree at least $10$, then $\chiup''(G) \leq \Delta(G) + 1$. 
\end{corollary}

\begin{corollary}%
If $G$ is a planar or toroidal graph with maximum degree at least $9$, then $\chiup''(G) \leq \Delta(G) + 2$. 
\end{corollary}


\begin{thebibliography}{10}

\bibitem{Behzad1965}
M.~Behzad, Graphs and their chromatic numbers, Ph.D. thesis, Michigan State
  University (1965).

\bibitem{MR977440}
O.~V. Borodin, On the total coloring of planar graphs, J. Reine Angew. Math.
  394 (1989) 180--185.

\bibitem{MR1483474}
O.~V. Borodin, A.~V. Kostochka and D.~R. Woodall, List edge and list total
  colourings of multigraphs, J. Combin. Theory Ser. B 71 (1997)~(2) 184--204.

\bibitem{MR1464341}
O.~V. Borodin, A.~V. Kostochka and D.~R. Woodall, Total colorings of planar
  graphs with large maximum degree, J. Graph Theory 26 (1997)~(1) 53--59.

\bibitem{MR2537481}
D.~Du, L.~Shen and Y.~Wang, Planar graphs with maximum degree 8 and without
  adjacent triangles are 9-totally-colorable, Discrete Appl. Math. 157
  (2009)~(13) 2778--2784.

\bibitem{MR2277573}
J.~Hou, G.~Liu and J.~Cai, List edge and list total colorings of planar graphs
  without 4-cycles, Theoret. Comput. Sci. 369 (2006)~(1-3) 250--255, claims.

\bibitem{MR2679093}
J.~Hou, J.~Wu, G.~Liu and B.~Liu, Total coloring of embedded graphs of maximum
  degree at least ten, Sci. China Math. 53 (2010)~(8) 2127--2133.

\bibitem{MR1304254}
T.~R. Jensen and B.~Toft, Graph coloring problems, Wiley-Interscience Series in
  Discrete Mathematics and Optimization, John Wiley \& Sons Inc., New York,
  1995.

\bibitem{MR1617950}
M.~Juvan, B.~Mohar and R.~{\v{S}}krekovski, List total colourings of graphs,
  Combin. Probab. Comput. 7 (1998)~(2) 181--188.

\bibitem{MR0453576}
A.~V. Kostochka, The total coloring of a multigraph with maximal degree {$4$},
  Discrete Math. 17 (1977)~(2) 161--163.

\bibitem{MR1425788}
A.~V. Kostochka, The total chromatic number of any multigraph with maximum
  degree five is at most seven, Discrete Math. 162 (1996)~(1-3) 199--214.

\bibitem{MR2448906}
{\L}.~Kowalik, J.-S. Sereni and R.~{\v{S}}krekovski, Total-coloring of plane
  graphs with maximum degree nine, SIAM J. Discrete Math. 22 (2008)~(4)
  1462--1479.

\bibitem{MR2907325}
R.~Li and B.~Xu, Edge choosability and total choosability of toroidal graphs
  without intersecting triangles, Ars Combin. 103 (2012) 109--118.

\bibitem{MR2458420}
B.~Liu, J.~Hou and G.~Liu, List edge and list total colorings of planar graphs
  without short cycles, Inform. Process. Lett. 108 (2008)~(6) 347--351.

\bibitem{MR2552636}
B.~Liu, J.~Hou, J.~Wu and G.~Liu, Total colorings and list total colorings of
  planar graphs without intersecting 4-cycles, Discrete Math. 309 (2009)~(20)
  6035--6043.

\bibitem{MR2155783}
R.~Luo and C.-Q. Zhang, Total chromatic number of graphs with small genus, in
  The {N}inth {Q}uadrennial {I}nternational {C}onference on {G}raph {T}heory,
  {C}ombinatorics, {A}lgorithms and {A}pplications, vol.~11 of \emph{Electron.
  Notes Discrete Math.}, Elsevier, Amsterdam, 2002, pp. 468--477.

\bibitem{MR1210116}
C.~J.~H. McDiarmid and A.~S{\'a}nchez-Arroyo, An upper bound for total
  colouring of graphs, Discrete Math. 111 (1993)~(1-3) 389--392.

\bibitem{MR0278995}
M.~Rosenfeld, On the total coloring of certain graphs, Israel J. Math. 9
  (1971)~(3) 396--402.

\bibitem{MR1781294}
D.~P. Sanders and J.~Maharry, On simultaneous colorings of embedded graphs,
  Discrete Math. 224 (2000)~(1-3) 207--214.

\bibitem{MR1684286}
D.~P. Sanders and Y.~Zhao, On total 9-coloring planar graphs of maximum degree
  seven, J. Graph Theory 31 (1999)~(1) 67--73.

\bibitem{MR2530185}
L.~Shen and Y.~Wang, Total colorings of planar graphs with maximum degree at
  least 8, Sci. China Ser. A 52 (2009)~(8) 1733--1742.

\bibitem{MR2475012}
X.-Y. Sun, J.-L. Wu, Y.-W. Wu and J.-F. Hou, Total colorings of planar graphs
  without adjacent triangles, Discrete Math. 309 (2009)~(1) 202--206.

\bibitem{MR0285447}
N.~Vijayaditya, On total chromatic number of a graph, J. London Math. Soc. (2)
  3 (1971)~(3) 405--408.

\bibitem{Vizing1965}
V.~G. Vizing, Critical graphs with given chromatic class, Metody Diskret.
  Analiz. 5 (1965) 9--17.

\bibitem{MR3146839}
H.~Wang, B.~Liu, J.~Wu and G.~Liu, Total coloring of embedded graphs with
  maximum degree at least seven, Theoret. Comput. Sci. 518 (2014) 1--9.

\bibitem{MR3146527}
H.~Wang, B.~Liu, J.~Wu and B.~Wang, Total coloring of graphs embedded in
  surfaces of nonnegative {E}uler characteristic, Sci. China Math. 57
  (2014)~(1) 211--220.

\bibitem{MR2118301}
P.~Wang and J.-L. Wu, A note on total colorings of planar graphs without
  4-cycles, Discuss. Math. Graph Theory 24 (2004)~(1) 125--135.

\bibitem{2012arXiv1206.3862W}
T.~Wang, Total coloring of 1-toroidal graphs of maximum degree at least 11 and
  no adjacent triangles, eprint arXiv:1206.3862  (2012).

\bibitem{MR2285452}
W.~Wang, Total chromatic number of planar graphs with maximum degree ten, J.
  Graph Theory 54 (2007)~(2) 91--102.

\bibitem{MR2464909}
J.~Wu and P.~Wang, List-edge and list-total colorings of graphs embedded on
  hyperbolic surfaces, Discrete Math. 308 (2008)~(24) 6210--6215.

\bibitem{MR0976059}
H.~P. Yap, Total colourings of graphs, Bull. London Math. Soc. 21 (1989)~(2)
  159--163.

\bibitem{MR1724810}
Y.~Zhao, On the total coloring of graphs embeddable in surfaces, J. London
  Math. Soc. (2) 60 (1999)~(2) 333--343.

\end{thebibliography}
\end{document}